\documentclass[12pt]{amsart}
\usepackage{amssymb}
\usepackage{amsfonts}
\usepackage{latexsym}
\usepackage{amscd}
\usepackage[mathscr]{euscript}

\usepackage[active]{srcltx}

\usepackage{enumerate}

\usepackage[utf8]{inputenc}
\usepackage[all,cmtip,2cell]{xy}

\usepackage[colorlinks  = true,
            citecolor   = blue,
            linkcolor   = blue,
            bookmarks   = true,
            linktocpage = true]{hyperref}

\usepackage{tikz}
\usetikzlibrary{matrix}

\numberwithin{equation}{section}

\theoremstyle{plain}
\newtheorem{lemma}{Lemma}[section]

\newtheorem{theorem}[lemma]{Theorem}

\newtheorem{proposition}[lemma]{Proposition}
\newtheorem{definition}[lemma]{Definition}
\newtheorem*{proposition*}{Proposition}
\newtheorem*{theorem*}{Theorem}
\newtheorem*{definition*}{Definition}
\newtheorem*{claim*}{Claim}
\newtheorem*{notation*}{Notation}

\newtheorem{remark}[lemma]{Remark}

\newcommand{\N}{{\mathbb{N}}}

\newcommand{\Z}{{\mathbb{Z}}}

\newcommand{\ol}{\overline}
\newcommand{\uloopr}[1]{\ar@'{@+{[0,0]+(-4,5)}@+{[0,0]+(0,10)}@+{[0,0] +(4,5)}}^{#1}}
\newcommand{\uloopd}[1]{\ar@'{@+{[0,0]+(5,4)}@+{[0,0]+(10,0)}@+{[0,0]+ (5,-4)}}^{#1}}
\newcommand{\dloopr}[1]{\ar@'{@+{[0,0]+(-4,-5)}@+{[0,0]+(0,-10)}@+{[0, 0]+(4,-5)}}_{#1}}
\newcommand{\dloopd}[1]{\ar@'{@+{[0,0]+(-5,4)}@+{[0,0]+(-10,0)}@+{[0,0 ]+(-5,-4)}}_{#1}}

\newcommand{\Typ}{\mbox{\rm Typ}}

\newcommand{\luloop}[1]{\ar@'{@+{[0,0]+(-8,2)}@+{[0,0]+(-10,10)}@+{[0, 0]+(2,2)}}^{#1}}

\newcommand{\AC}[2]{G_{#1}[#2]}

\newcommand{\Ifree}{I_{\mathrm{free}}}
\newcommand{\Ireg}{I_{\mathrm{reg}}}

\DeclareMathOperator{\rL}{L}

\newdimen \boxht

\begin{document}
\title[Refinement Monoids and Adaptable Separated Graphs]{Refinement Monoids and adaptable Separated graphs}

\author{Pere Ara}
\address{Departament de Matem\`atiques, Universitat Aut\`onoma de Barcelona,
08193 Bellaterra (Barcelona), Spain, and Barcelona Graduate School of Mathematics (BGSMath).} \email{para@mat.uab.cat}
\author{Joan Bosa}
\address{Departament de Matem\`atiques, Universitat Aut\`onoma de Barcelona,
08193 Bellaterra (Barcelona), Spain, and Barcelona Graduate School of Mathematics (BGSMath).} \email{jbosa@mat.uab.cat}
\author{Enrique Pardo}
\address{Departamento de Matem\'aticas, Facultad de Ciencias, Universidad de C\'adiz,
Campus de Puerto Real, 11510 Puerto Real (C\'adiz),
Spain.}\email{enrique.pardo@uca.es}\urladdr{https://sites.google.com/a/gm.uca.es/enrique-pardo-s-home-page/}
%\author{Aidan Sims}
%\address{School of Mathematics and Applied Statistics, University of Wollongong, Wollongong NSW 2522, Australia.} \email{asims@uow.edu.au}

\thanks{First, second and third authors are partially supported by the DGI-MINECO and European Regional Development Fund, jointly, through grants MTM2014-53644-P and MTM2017-83487-P. First and second author acknowledge support from the Spanish Ministry of Economy and Competitiveness, through the María de Maeztu Programme for Units of Excellence in R$\&$D (MDM-2014-0445). Third author was partially supported by PAI III grant FQM-298 of the Junta de Andaluc\'{\i}a.}
\date{\today}
%
%\dedicatory{Preprint, version 1.0}
%\commby{}

\begin{abstract}
We define a subclass of separated graphs, the class of {\it adaptable separated graphs}, and study their associated monoids. We show that these monoids are primely generated conical refinement monoids, 
and we explicitly determine their associated $I$-systems.  We also show that any finitely generated conical refinement monoid can be represented as the monoid of an adaptable separated graph.
These results provide the first step toward an affirmative answer to the Realization Problem for von Neumann regular rings,  in the finitely generated case.
\end{abstract}

\maketitle

\section*{Introduction.}\label{Sect:Intro}
The structure of commutative refinement monoids is generally very intricate, and it is difficult to rephrase their architecture in terms of combinatorial data.
These monoids appear naturally in different contexts, such as non-stable K-theory of exchange rings and real rank zero $C^*$-algebras (see e.g. \cite{AGOP, OPR}), classification of Boolean algebras
(see e.g. \cite{Ket, Pierce}), 
the realization problem for von Neumann regular rings (see below), and the theory of type semigroups (see e.g. \cite{RS,Weh}). In this paper, based on the work developed in \cite{AP16} and \cite{AP17}, 
we provide a concrete and useful description of a subclasss of all primely generated conical refinement monoids, which contains all the finitely generated ones, in terms of a 
specific type of separated graphs. 

Recall that a separated graph \cite{AG12} is a pair $(E,C)$, where $E$ is a directed graph and $C$ is a partition of the set of edges of $E$ which is finer than the partition induced by the source map
$s\colon E^1\to E^0$. Visually one may think of a separated graph as a directed graph where the edges have been given different colours. Several interesting algebras and $C^*$-algebras have been attached to these 
combinatorial objects, some of them having exotic behaviour (see for instance \cite{AE, AG12}). Given a separated graph $(E,C)$, one can naturally associate a monoid $M(E,C)$ to it \cite{AG12}. 
However, it is not always true that $M(E,C)$ is a refinement monoid \cite[Section 5]{AG12}. 

Generalizing earlier work by Dobbertin \cite{Dobb84} and Pierce \cite{Pierce}, the first and third-named authors have completely determined in \cite{AP16} the structure of primely generated conical refinement monoids. The main ingredient of 
this characterization is the notion of an $I$-system, which is a certain poset of semigroups generalizing the posets of groups used by Dobbertin in \cite{Dobb84} (see Definition \ref{def:I-system} below). 
Using this description, a characterization of the finitely generated conical refinement monoids which are isomorphic to a graph monoid $M(E)$ for a (non-separated) directed graph $E$ has been obtained in \cite{AP17}.
In particular, we stress the fact that {\it not} all such monoids are isomorphic to graph monoids. It is the purpose of this paper to show that a large class of primely generated conical refinement monoids, including all the 
finitely generated ones, can be obtained as monoids of the form $M(E,C)$ for $(E,C)$ belonging to a particularly well-behaved class of separated graphs, the {\it adaptable separated graphs} 
(see Definition \ref{def:adaptable-sepgraphs} below).

Concretely, the main result of this paper (Theorem \ref{thm:main}) is the following:

\begin{theorem*}
%\label{thm:main} 
	The following two statements hold:
	\begin{enumerate}
		\item If $(E,C)$  is an adaptable separated graph, then $M(E,C)$ is a primely generated conical refinement monoid.
		\item For any finitely generated conical refinement monoid $M$, there exists an adaptable separated graph $(E,C)$ such that $M\cong M(E,C)$.  
	\end{enumerate}
\end{theorem*}

We now outline some of the applications of the results obtained in this note. Concretely, we use the structure of an adaptable separated graph in order to get two realization results. 
The first application is given in \cite{ABPS}, where the authors, jointly with A. Sims, attach to each adaptable separated graph $(E,C)$ an $E^*$-unitary inverse semigroup $S(E,C)$.
Moreover, using techniques developed by Paterson \cite{Paterson} and Exel \cite{ExelBraz}, they build from this inverse semigroup $S(E,C)$ an
ample Hausdorff \'etale topological groupoid $\mathcal G (E,C)$ satisfying
$$\Typ (\mathcal G (E,C)) \cong M(E,C).$$
In particular, we see from Theorem \ref{thm:main}(2) that all finitely generated conical refinement monoids arise as type semigroups of this well-behaved class of topological groupoids.
The second application concerns the {\it Realization Problem for von Neumann regular rings}, posed by Goodearl in \cite{directsum}. 
This wonders which refinement monoids appear as a $\mathcal V(R)$ for a von Neumann regular ring $R$, where the latter stands for the monoid of isomorphism classes of finitely generated 
projective (left, say) $R$-modules, with the operation induced from direct 
sum (see \cite{Areal} for a survey on this problem).
For an adaptable separated graph $(E,C)$ and an arbitrary field $K$, we build in \cite{ABP} a von Neumann regular $K$-algebra $Q_K(E,C)$, which is a certain universal localization 
of the Steinberg algebra $A_K(\mathcal G(E,C))$ of the above groupoid $\mathcal G(E,C)$, and which satisfies that
$$\mathcal V (Q_K(E,C)) \cong M(E,C).$$
Again, Theorem \ref{thm:main}(2) gives that the realization problem for von Neumann regular $K$-algebras has a positive answer for any finitely generated conical refinement monoid. 
This construction extends at once the constructions given in \cite{Ara10} and \cite{AB}.

 The paper is organized as follows. 
 In the first section we introduce background material needed for our results. We have splitted this in three subsections, concerning commutative monoids, primely generated refinement monoids, and
 separated graphs, respectively. In Section 2, we prove our results. We have divided this section into two subsections, in each of which we prove one of the statements of our
 Theorem. 

%%%%%%%%%%%%%%%%%%%
\section{Preliminaries}\label{sec:Preliminaries}
\subsection{Basics on commutative monoids.}
All semigroups and monoids considered in this paper are commutative. We will denote by $\N$ the semigroup of positive integers, and by $\Z^+$ the monoid of non-negative integers.  

Given a commutative monoid $M$, we set $M^*:=M\setminus\{0\}$. We say that $M$ is {\it conical} if $M^*$ is a semigroup, that is, if, for all $x$, $y$ in $M$, $x+y=0$ only when $x=y=0$. 

We say that a monoid $M$ is {\it separative} 
provided $2x=2y=x+y$ always implies $x=y$; there are a number of equivalent formulations of this property, see e.g. \cite[Lemma 2.1]{AGOP}.
We say that  $M$ is
a {\it refinement monoid} if, for all $a$, $b$, $c$, $d$ in
$M$ such that $a+b=c+d$, there exist $w$, $x$, $y$, $z$ in $M$ such
that $a=w+x$, $b=y+z$, 
$c=w+y$ and $d=x+z$.  %It will often be convenient to present this
%situation in the form of a diagram, as follows:
%$$\mbox{\begin{tabular}{|l|l|l|}
%\cline{2-3}
%\multicolumn{1}{l|}{} & ${c}$ & ${d}$ \\ \hline
%${a}$ & ${w}$ & ${x}$ \\ \hline
%${b}$ & ${y}$ & ${z}$ \\ \hline
%\end{tabular}}$$  
A basic example of refinement monoid is the monoid $M(E)$ associated to a countable row-finite graph $E$ \cite[Proposition 4.4]{AMFP}.

If $x, y\in M$, we write $x\leq y$ 
if there exists $z\in M$  such that $x+z = y$.
Note that $\le$ is a translation-invariant pre-order on $M$, called the {\it algebraic pre-order} of $M$. All inequalities in commutative monoids will be with respect to this pre-order. 
An element $p$ in a monoid $M$ is a {\it prime element} if $p$ is not invertible in $M$, and, whenever 
$p\leq a+b$ for $a,b\in M$, then either $p\leq a$ or $p\leq b$. The monoid $M$ is {\it primely generated} if every non-invertible element of $M$ 
can be written as a sum of prime elements.
 
An element $x\in M$ is {\it regular} if $2x\leq x$. An element $x\in M$ is an {\it idempotent} if $2x= x$. An element 
$x\in M$ is {\it free} if $nx\leq mx$ implies $n\leq m$. Any element of a separative monoid is either free or regular.
In particular, this is the case for any primely generated refinement monoid, by \cite[Theorem 4.5]{Brook}. Furthermore, 
every finitely generated refinement monoid is primely generated \cite[Corollary 6.8]{Brook}. 

A subset $S$ of a monoid $M$ is called an {\it order-ideal} if $S$ is a subset of $M$ containing $0$,
closed under taking sums and summands within $M$.  An order-ideal can also be described as a submonoid $I$ of
$M$, which is hereditary with respect to the canonical pre-order
$\le $ on $M$: $x\le y$ and $y\in I$ imply $x\in I$. A non-trivial monoid is said to be {\it simple} if it has no non-trivial order-ideals.

If $(S_k)_{k\in \Lambda}$ is a family of (commutative) semigroups, $\bigoplus _{k\in \Lambda} S_k$ (resp. $\prod _{k\in \Lambda} S_k$) 
stands for the coproduct (resp. the product) of the semigroups $S_k$, $k\in \Lambda$, in the category of commutative
semigroups. If the semigroups $S_k$ are subsemigroups of a semigroup $S$, we will denote by $\sum_{k\in \Lambda} S_k$ the subsemigroup of $S$ generated by $\bigcup_{k\in \Lambda}S_k$.
Note that $\sum_{k\in \Lambda} S_k$ is the image of the canonical map $\bigoplus_{k\in \Lambda} S_k\to S$. We will use the notation $\langle X\rangle $ to denote the semigroup
generated by a subset $X$ of a semigroup $S$.

Given a semigroup $M$, we will denote by $G(M)$ the Grothendieck group of $M$. There exists a semigroup homomorphism $\psi_M\colon M\to G(M)$
such that for any semigroup homomorphism $\eta \colon M\to H$ to a group $H$ there is a unique group homomorphism $\widetilde{\eta}\colon G(M)\to H$ such that 
$\widetilde{\eta}\circ \psi_M= \eta$. $G(M)$ is abelian and it is generated as a group by $\psi (M)$.  If $M$ is already a group then $G(M)= M$. If $M$ is a semigroup of the form 
$\N\times G$, where $G$ is an abelian group, then $G(M)= \Z\times G$. In this case, we will view $G$ as a subgroup of $\Z\times G$ by means of the identification $g\leftrightarrow (0,g)$. 

Let $M$ be a conical commutative monoid, and let $x\in M$ be any element. The {\it archimedean component} of $M$ generated by $x$ is the subsemigroup
$$G_M[x]:=\{a\in M : a\leq nx \text{ and } x\leq ma \text{ for some } n,m\in \N\}.$$

For any $x\in M$, $G_M[x]$ is a simple semigroup. If $M$ is separative, then $G_M[x]$ is a cancellative semigroup; if moreover $x$ is a regular element, 
then $G_M[x]$ is an abelian group.

%%%%%%%%%%%%%%%%%%%

\subsection{Primely generated refinement monoids}
\label{subsec:Primely-gen}

The structure of primely generated refinement monoids has been recently described in \cite{AP16}. We recall here some basic facts.

Given a poset $(I, \leq)$, we say that a subset $A$ of $I$ is a {\it lower set} if $x\leq y$ in $I$ and $y\in A$ implies $x\in A$. For any $i\in I$, 
we will denote by $I\downarrow i=\{x\in I : x\leq i\}$ the lower subset generated by $i$. 
We will write $x<y$ if $x\le y$ and $x\ne y$. 

The following definition is crucial for this work:

\begin{definition}[{\cite[Definition 1.1]{AP16}}]
\label{def:I-system} {\rm Let $I= (I,\le )$ be a poset.  An {\it $I$-system} $$\mathcal{J}=
\left(I, \leq , (G_i)_{i\in I}, \varphi_{ji} \, (i<j)\right)$$ is given by the following data:
\begin{enumerate}
\item[(a)] A partition
$I=I_{free}\sqcup I_{reg}$ (we admit one of the two sets
$I_{free}$ or $I_{reg}$ to be empty).
\item[(b)] A family $\{G_i\}_{i\in I}$ of abelian groups. We adopt the following notation: 
\begin{itemize}
\item[(1)] For $i\in I_{reg}$, set $M_i = G_i$, and $\widehat{G}_i=G_i=M_i$.
\item[(2)] For $i\in I_{free}$, set $M_i=\N \times G_i$, and $\widehat{G}_i= \Z\times G_i$
\end{itemize}
Observe that, in any case, $\widehat{G}_i$ is the Grothendieck group of $M_i$.
\item[(c)]
A family of  semigroup homomorphisms $\varphi _{ji}\colon M_i\to
G_j$ for all $i<j$, to which we associate, for all $i<j$, the unique extension $\widehat{\varphi}_{ji}\colon \widehat{G}_i \to G_j$ of $\varphi _{ji}$ to a group homomorphism
from the Grothendieck group of $M_i$  to $G_j$ (we look at these
maps as maps from $\widehat{G}_i$ to $\widehat{G}_j$). We require that the family $ \{ \varphi_{ji} \}$ satisfies the following conditions:
\begin{itemize}
\item[(1)]  The assignment
$$
\left\{
\begin{array}{ccc}
i & \mapsto  &  \widehat{G}_i   \\
 (i<j) & \mapsto  &   \widehat{\varphi}_{ji}
\end{array}
\right\}
$$ 
defines a functor from the
category $I$ to the category of abelian groups (where we set $\widehat{\varphi}_{ii}= {\rm  id}_{\widehat{G}_i}$
for all $i\in I$).
\item[(2)]  For each $i\in I_{free}$ we have that the map
$$\bigoplus _{k<i} \varphi _{ik}\colon  \bigoplus _{k<i}  M_k \to G_i$$
is surjective.
\end{itemize}
\end{enumerate}
We say that an $I$-system $\mathcal  J =  \left(I, \leq , (G_i)_{i\in I}, \varphi_{ji}\, (i<j)\right)$ is {\it finitely generated} 
in case $I$ is a finite poset and all the groups $G_i$ are finitely generated.} 
\end{definition}

To every $I$-system $\mathcal J$ one can associate a primely generated conical refinement monoid $M(\mathcal J)$, and conversely to any primely generated conical refinement monoid $M$, we can associate
an $I$-system $\mathcal J$ such that $M\cong M(\mathcal J)$, see Sections 1 and 2 of \cite{AP16} respectively.

%%%%%%%%%%%%%%%%%%%

\subsection{Separated graphs}

Here, we recall definitions and properties about separated graphs that will be needed in the sequel. In particular, we define the notion of {\it adaptable} separated graph, 
which is crucial for this paper. We refer the reader to \cite{AAS} and \cite{AG12} for more information and general notation about (separated) graphs.

Let $E$ be a directed graph,
and let $\le $ be the preorder on $E^0$ determined by $w\ge v$ if
there is a path in $E$ from $w$ to $v$. Let $I$ be the
antisymmetrization of $E^0$, with the partial order $\le $ induced by
the order on $E^0$. Thus, denoting by $[v]$ the class of $v\in E^0$
in $I$, we have $[v]\le [w]$ if and only if $v\le w$.

For $v\in E^0$, we refer to the set $[v]$ as the component
of $v$, and we will denote by $E[v]$ the restriction of $E$ to $[v]$, that is, the graph with
$E[v]^0= [v]$ and $E[v]^1= \{e\in E^1\mid s(e)\in [v] \text{ and } r(e)\in [v] \}$.
If $J$ is a lower subset of $I$, we will denote by $E|_J$ the restriction of the graph $E$ to the set
of vertices $\{ v\in E^0 \mid [v]\in J \}$.

We now describe our graphs.

\begin{definition}[{\cite[Definition 2.1]{AG12}}]\label{defsepgraph}
{\rm A \emph{separated graph} is a pair $(E,C)$ where $E$ is a directed graph,  $C=\bigsqcup
_{v\in E^ 0} C_v$, and
$C_v$ is a partition of $s^{-1}(v)$ (into pairwise disjoint nonempty
subsets) for every vertex $v$. (In case $v$ is a sink, we take $C_v$
to be the empty family of subsets of $s^{-1}(v)$).

If all the sets in $C$ are finite, we shall say that $(E,C)$ is a \emph{finitely separated} graph.}
\end{definition}

From now on, we will assume that all our separated graphs are finitely separated graphs without any further comment.

Following \cite{AG12}, we associate the following monoid to any finitely separated graph.

\begin{definition}[{\cite[Definition 4.1]{AG12}}]\label{defsepgraphmonoid}
Given a finitely separated graph $(E,C)$, we define the monoid of the separated graph $(E,C)$, to be 
\begin{equation}\label{(M)}
M(E,C)=\Big{\langle} a_v \,\, \, (v\in E^0) \, : a_v=\sum _{\{ e\in X\}}a_{r(e)} \text{ for every } X\in C_v, v\in E^0\Big{\rangle} .\end{equation}
\end{definition}

Recall that a directed graph is said to be {\it transitive} if any two vertices can be connected by a finite directed path. 

\begin{definition}
    \label{def:adaptable-sepgraphs}
    {\rm Let $(E,C)$ be a finitely separated graph and let $(I,\le )$ be the antisymmetrization of $(E^0,\leq)$. We say that $(E,C)$ is {\it adaptable} if $I$ is finite,
    and there exist a partition $I=\Ifree \sqcup \Ireg $, and a family of subgraphs $\{ E_p \}_{p\in I}$ of $E$ such that the following conditions are satisfied:
        \begin{enumerate}
            \item $E^0=\bigsqcup_{p\in I} E_p^0$, where $E_p$ is a transitive row-finite graph if $p\in \Ireg$ and $E_p^0= \{ v^p \}$ is a single vertex if $p\in \Ifree $.
            \item For $p\in \Ireg$ and $w\in E_p^0$, we have that $|C_w|= 1$ and $|s_{E_p}^{-1} (w)|\ge 2$. Moreover, all edges departing from $w$ either belong to the graph $E_p$ or connect $w$ to a vertex $u\in E_q^0$,
            with $q<p$ in $I$.
            \item For $p\in \Ifree$, we have that $s^{-1}(v^p) = \emptyset $ if and only if $p$ is minimal in $I$. If $p$ is not minimal, then there is a positive integer $k(p)$ such that
            $C_{v^p}=\{ X^{(p)}_1,\dots ,X^{(p)}_{k(p)} \}$. Moreover, each $X^{(p)}_i$ is of the form
            $$X^{(p)}_i = \{ \alpha (p,i) ,\beta (p,i,1),\beta (p,i,2),\dots , \beta(p, i, g(p,i)) \} ,$$
            for some $g(p,i) \ge 1$, where $\alpha (p,i)$ is a loop, i.e., $s(\alpha(p,i))= r(\alpha (p,i)) = v^p$, and $r(\beta (p,i,t))\in E^0_q$ for $q<p$ in $I$.
            Finally, we have $E_p^1= \{ \alpha(p,1),\dots ,\alpha (p,k(p)) \}$.
        %   \item There is a monoid isomorphism $\delta \colon M(E,C) \to M(\mathcal J)$ sending the element $v^p\in M(E,C)$ to the element $p\in M(\mathcal J)$ if $p$ is a free prime,
        %   and sending the vertices $v\in E_p^0$ to a family of generators of the group $G_p$ if $p$ is a regular prime.
        \end{enumerate}

        The edges connecting a vertex $v\in E_p^0$ to a vertex $w\in E_q^0$ with $q<p$ in $I$ will be called {\it connectors}.} \qed
    %In particular, we will denote them as $\beta(v,p,q,w)$.
\end{definition}

\section{Adaptable separated graphs and their associated monoids.}\label{Section1}
In this section we show the main result of the paper:

\begin{theorem}\label{thm:main} 
	The following two statements hold:
\begin{enumerate}
\item If $(E,C)$  is an adaptable separated graph, then $M(E,C)$ is a primely generated conical refinement monoid.
\item For any finitely generated conical refinement monoid $M$, there exists an adaptable separated graph $(E,C)$ such that $M\cong M(E,C)$. 
\end{enumerate}
	\end{theorem}

 We have divided the proof in two parts. First we show statement {\rm (1)} (Proposition \ref{prop:MECisprimelygenerated}), and, subsequently, we show the realization result stated in {\rm (2)} (Theorem \ref{thm:realM-by-sepgraph}).

%%%%%%%%%%%%%%%%%%%%%%%%%%%%%%%%%%%%%%%%%%%%%%%%%%%%%%%%%%%%%%%%55
%%%%%%%%%%%%%%%%%%%%%%%%%%%%%%%%%%%%%%%%%%%%%%%%%%%%%%%%%%%%%%%%%55

\subsection{The monoid of an adaptable separated graph}

We show below that the monoid $M(E,C)$ associated to an adaptable separated graph $(E,C)$ is a primely generated conical refinement monoid. 
As a consequence, we obtain from \cite[Theorem 2.7]{AP16} that there is a poset $\mathbb P$, with a partition $\mathbb P = \mathbb P_{{\rm free}} \sqcup \mathbb P_{{\rm reg}}$, and a $\mathbb P$-system 
$\mathcal J$ such that $M(E,C)\cong M(\mathcal J)$. We will explicitly determine this system. 

To show our results, we will need the ``confluence'' property of the congruence associated to our separated graphs $(E,C)$. This was established for all graph monoids $M(E)$ of ordinary
row-finite graphs in \cite[Lemma 4.3]{AMFP}. Amongst other things, this enables us to show the refinement property of the monoids $M(E,C)$, when $(E,C)$ is an adaptable separated graph.

Let $(E,C)$ be an adaptable separated graph, and $F$ be the free commutative monoid on the set $E^0$. The nonzero
elements of $F$ can be written in a unique form up to permutation
as $\sum _{i=1}^n v_i$, where $v_i\in E^0$. Now we will give a
description of the congruence on $F$ generated by the relations
(\ref{(M)}) (see Definition \ref{defsepgraphmonoid}) on $F$.

 It will be convenient to introduce the
following notation. For $X\in C_v$ ($v\in E^0$), write
$${\bf r}(X):=\sum _{e\in X} r(e)\in F .$$
With this new notation, the relations in (\ref{(M)}) become $v={\bf r}(X)$ 
for every $v\in E^0$ and every $X\in C_v$.

\begin{definition}
	\label{def:binary} {\rm Define a binary relation $\rightarrow_1$ on
		$F\setminus \{0\}$ as follows. Let $\sum _{i=1}^n v_i\in F\setminus \{ 0 \}$, and let $X\in C_{v_j}$ for some 
		$j\in \{ 1,2,\dots ,n \}$. Then $\sum _{i=1}^n v_i\rightarrow_1
		\sum _{i\ne j}v_i+{\bf r}(X)$. Let $\rightarrow $ be the
		transitive and reflexive closure of $\rightarrow _1$ on
		$F\setminus \{0\}$, that is, $\alpha\rightarrow \beta$ if and only
		if there is a finite string $\alpha =\alpha _0\rightarrow _1
		\alpha _1\rightarrow _1 \cdots \rightarrow _1 \alpha _t=\beta.$
		
		Let $\sim$ be the congruence on $F$ generated by the relation
		$\rightarrow_1 $ (or, equivalently, by the relation $\rightarrow
		$). Namely $\alpha\sim \alpha$ for all $\alpha \in F$ and, for
		$\alpha,\beta \ne 0$, we have $\alpha\sim \beta$ if and only if
		there is a finite string $\alpha =\alpha _0,\alpha_1, \dots
		,\alpha _n=\beta$, such that, for each $i=0,\dots ,n-1$, either
		$\alpha _i\rightarrow _1 \alpha _{i+1}$ or
		$\alpha_{i+1}\rightarrow_1 \alpha _i$. The number $n$ above will
		be called the {\it length} of the string.}\qed
\end{definition}

It is clear that $\sim $ is the congruence on $F$ generated by
relations (\ref{(M)}), and so $M(E,C) = F/{\sim}$.

The {\em support} of an element $\gamma$ in $F$, denoted
$\mbox{supp}(\gamma)\subseteq E^0$, is the set of basis elements
appearing in the canonical expression of $\gamma$.

The proof of the following easy lemma is similar to the one of \cite[Lemma 4.2]{AMFP}.

\begin{lemma} (cf. \cite[Lemma 4.2]{AMFP}\label{division})
	\label{lem:division} Let $(E,C)$ be any finitely separated graph. 
	Let $\rightarrow $ be the binary relation on $F$
	defined above and $\alpha,\beta\in F\setminus\{0\}$. Assume that $\alpha =\alpha _1+\alpha _2$ and
	$\alpha \rightarrow \beta $. Then $\beta $ can be written as
	$\beta =\beta _1+\beta _2$, with $\alpha _1\rightarrow \beta _1$
	and $\alpha_2\rightarrow \beta _2$.
\end{lemma}

We are now ready to obtain the crucial lemma that gives the important ``confluence" property of
the congruence $\sim$ on the free commutative monoid $F$.

\begin{lemma}
	\label{lem:confluence} Let $(E,C)$ be an adaptable separated graph. 
	Let $\alpha$ and $\beta$ be nonzero elements in
	$F$. Then $\alpha \sim \beta$ if and only if there is $\gamma\in
	F$ such that $\alpha \rightarrow \gamma$ and $\beta \rightarrow
	\gamma$.
\end{lemma}

\begin{proof}
	The proof is similar to the proof of \cite[Lemma 4.3]{AMFP}.
	We highlight the point in which both proofs differ. 
	
	Assume that $\alpha \sim \beta$. Then there exists a finite string
	$\alpha =\alpha _0,\alpha_1, \dots ,\alpha _n=\beta$ such that,
	for each $i=0,\dots ,n-1$, either $\alpha _i\rightarrow _1 \alpha
	_{i+1}$ or $\alpha_{i+1}\rightarrow_1 \alpha _i$. We proceed by
	induction on $n$. If $n=0$, then $\alpha=\beta$ and there is
	nothing to prove. Assume the result is true for strings of length
	$n-1$, and let $\alpha =\alpha_0,\alpha_1, \dots ,\alpha_n=\beta$
	be a string of length $n$. By induction hypothesis, there is
	$\lambda\in F$ such that $\alpha\rightarrow \lambda$ and
	$\alpha_{n-1}\rightarrow \lambda$. Now there are two cases to
	consider. If $\beta\rightarrow_1 \alpha_{n-1}$, then $\beta
	\rightarrow \lambda$ and we are done. Assume that
	$\alpha_{n-1}\rightarrow _1 \beta$. By definition of $\rightarrow
	_1$, there is a basis element $v\in E^0$ in the support of $\alpha
	_{n-1}$  and $X\in C_v$ such that $\alpha _{n-1}=v+\alpha_{n-1}'$ and $\beta =
	{\bf r}(X)+\alpha _{n-1}'$. By Lemma \ref{division}, we have
	$\lambda = \lambda (v)+\lambda '$, where $v\rightarrow \lambda
	(v)$ and $\alpha_{n-1}'\rightarrow \lambda '$. If the length of
	the string from $v$ to $\lambda (v)$ is positive, then we have
	${\bf r}(Y)\rightarrow \lambda (v)$ for some $Y\in C_v$.
	If $[v]\in \Ireg $, then $X=Y$ and the proof continues as in \cite[Lemma 4.3]{AMFP}.
	If $[v]\in \Ifree $, then $X$ may be distinct from $Y$, but in this case we play with the special form of the sets in $C_v$.
	Indeed, assume that $[v]\in \Ifree$.  
	In this case, write $\lambda '' := \lambda + ({\bf r}(X)-v)$.
	Then we have
	\begin{align*}
	\beta &  = {\bf r}(X) + \alpha_{n-1}'  = v+  ({\bf r}(X)-v) + \alpha_{n-1}' \\
	& \rightarrow _1 {\bf r} (Y) + ({\bf r}(X)-v) + \alpha_{n-1}' \\
	& \rightarrow \lambda (v) + \alpha'_{n-1} + ({\bf r}(X) -v) \\
	&  \rightarrow \lambda (v)+ \lambda' + ({\bf r}(X)-v) \\
	& = \lambda + ({\bf r}(X)-v) = \lambda ''.
	\end{align*}
	On the other hand, since $v+\alpha _{n-1}'\to \lambda$ and since $[v] \in \Ifree$, it follows easily by induction
	on the length of this string that $v\in \mbox{supp}(\lambda)$ and thus $\lambda \to_1 \lambda + ({\bf r}(X) - v) = \lambda ''$.
	Hence $\alpha \to \lambda \to \lambda ''$ and $\beta \to \lambda''$, as desired.

	In the remaining case
	that $v=\lambda (v)$, set $\gamma= {\bf r}(X) +\lambda '$.
	Then we have $\lambda \rightarrow _1 \gamma $ and so $\alpha
	\rightarrow \gamma $, and also $\beta ={\bf r}(X) +\alpha
	_{n-1}'\rightarrow {\bf r}(X)+\lambda '=\gamma $. This concludes
	the proof.
\end{proof}

Now, exactly the same proof as in \cite[Proposition 4.4]{AMFP} (using Lemmas \ref{division} and \ref{lem:confluence}) gives the following result.

\begin{proposition}
	\label{prop:refinement} Let $(E,C)$ be an adaptable separated graph. Then the monoid $M(E,C)$ is a refinement monoid.
\end{proposition}

We now show that, for any adaptable separated graph, the monoid $M(E,C)$ is a primely generated monoid.

\begin{proposition}
	\label{prop:MECisprimelygenerated}
	Let $(E,C)$ be an adaptable separated graph and let $(I,\le )$ be the antisymmetrization of $E^0$ with respect to the path-way pre-order. Then 
	$M(E,C)$ is a primely generated conical refinement monoid.
	\end{proposition}

\begin{proof} By \cite[Lemma 4.2]{AG12}, $M(E,C)$ is a nonzero, conical monoid whenever $(E,C)$ is an arbitrary finitely separated graph such that $E^0$ is non-empty.

	Suppose now that $(E,C)$ is an adaptable separated graph. By Proposition \ref{prop:refinement}, $M(E,C)$ is a refinement monoid.  
	We now show that $M(E,C)$ is primely generated. For this, it is enough to observe that each generator $a_v$, with $v\in E^0$ is prime in $M(E,C)$.
	For this purpose, we work in the free monoid $F$ generated by $E^0$ and we use the notation introduced above. We have to show that if we have a relation $[v]+[\delta] = [\alpha_1] + [\alpha_2]$ in $F/{\sim} =M(E,C)$, 
	then there is $i\in \{ 1,2 \}$ such that $[v]\le [\alpha_i]$. Now since $v+\delta \sim  \alpha_1 + \alpha_2$ in $F$, we have by Lemma \ref{lem:confluence} that there is $\gamma \in F$ such that
	$v+\delta \to \gamma $ and $\alpha _1+\alpha_2 \to \gamma $. By Lemma \ref{lem:division}, we can write $\gamma = \gamma_1 + \gamma_2$ with $\alpha_i \to \gamma_i$ for $i=1,2$. Since 
	$v+\delta \to \gamma $ and, by the definition of an adaptable separated graph, each $X\in C$ contains at least a loop, we see that $v$ belongs to the support of $\gamma $. Therefore, it belongs to the support of 
	$\gamma _i$ for some $i\in \{ 1,2 \}$. We can thus assume that $\gamma_1= v+\gamma _1'$ and therefore
	$$[\alpha_1] = [\gamma _1]= [v] + [\gamma _1'] \, ,$$
	showing that $[v]\le [\alpha_1]$, as desired. 
\end{proof}

It follows from Proposition \ref{prop:MECisprimelygenerated} and \cite[Theorem 2.7]{AP16} that for any adaptable separated graph $(E,C)$ there exists a poset $\mathbb P$, a partition
$\mathbb P =\mathbb P_{{\rm free}} \sqcup \mathbb P_{{\rm reg}}$, and a $\mathbb P$-system
$\mathcal J$ such that $M(E,C)\cong M(\mathcal J)$. We close this subsection by explicitly computing this system. Together with our main result in the next subsection (Theorem \ref{thm:realM-by-sepgraph}), 
this allows us to express all the 
structure of a finitely generated conical refinement monoid in terms of the information contained in a representing adaptable separated graph. 

Let $(E,C)$ be an adaptable separated graph and let $(I,\le )$ be the antisymmetrization of $E^0$ with respect to the path-way pre-order. In order to neatly express our result, we first define a certain $I$-system
and then we will show it is isomorphic to the system corresponding to $M(E,C)$.

\begin{definition}
 \label{def:Isystemofadpasepgraph}
 {\rm Let $(E,C)$ be an adaptable separated graph, let $(I,\le )$ be the antisymmetrization of $E^0$, and let 
 $I=\Ifree \sqcup \Ireg$ be the canonical partition of $I=E^0/{\sim}$ (see Definition \ref{def:adaptable-sepgraphs}). Define an $I$-system 
 $\mathcal J ''= (I,\le , (G_p'')_{p\in I}, \varphi''_{p,q} \, (q<p))$ as follows:
 
\noindent $(1)$ For each $p\in \Ifree$ minimal, define $G''_p:=\{0\} \,\,(i.e.\,\, M_p=\mathbb N)$.  
	Now for each non-minimal $p\in \Ifree$ , consider the abelian group $G''_p$ generated by elements $x^p_w$, where $w$ is a vertex in $E$ such that $[w]<p=[v^p]$, subject to the 
	relations 
	\begin{equation}
	\label{eq:relation-for-w-regular}
	x^p_w=\sum_{e\in s_E^{-1}(w)} x^p_{r(e)},\qquad [w]\in \Ireg, 
	\end{equation}
and 
\begin{equation}
	\label{eq:relation-for-w-free}
	\sum_{j=1}^{g(q,i)} x_{r(\beta (q,i,j))}^p =0,\qquad (i=1,\dots , k(q)) \quad \text{for } q\in \Ifree, q\le p .
	\end{equation}
$(2)$ For $p\in \Ireg$, we let $G''_p$ be the abelian group with generators $x^p_w$, where $w$ is a vertex in $E$ such that $[w]\le p$, and with relations \eqref{eq:relation-for-w-regular} for every $w\in E^0$
	such that $[w]\in \Ireg$ and $[w]\le p$, and \eqref{eq:relation-for-w-free} for every $q\in \Ifree$ (note that in the latter case, $q<p$ for any $q\in \Ifree$, because $p\in \Ireg$). 
	
Recalling that $M''_p=G''_p$ if $p\in \Ireg$ and $M''_p= \N \times G''_p$ is $p\in \Ifree $, we now define the connecting homomorphisms $\varphi''_{p,q}\colon M''_q\to G''_p$, for $q<p$, as follows:  
 $$\varphi''_{p,q} (x_w^q) = x_w^p, \qquad \text{if } q\in \Ireg, $$
         and 
         $$\varphi'' _{p,q} (n, \sum_{w<v^q} c_wx_w^q) = nx^p_{v^q} + \sum_{w<v^p} c_wx_w^p , \qquad \text{if } q\in \Ifree  $$
	where $n\in \N$, and $c_w\in \Z$ are almost all $0$. It is straightforward to show that $\mathcal J''$ is an $I$-system.}
  \end{definition}

\begin{remark}
 \label{rem:grothendieck-for-regulars}
 {\rm 
 Note that, in case $p\in \Ireg$, the relations in $G_p''$ can be expressed in the form $x^p_w = \sum _{e\in X} x^p_{r(e)}$, for each $X\in C_w$ and each $w\in E^0$ such that $[w]\le p$. 
 The resulting group is therefore the Grothendieck group
of the monoid $M(E_H, C^H)$, where $(E_H, C^H)$ is the restriction of the separated graph $(E,C)$ to the hereditary set $H := \{ w\in E^0: [w]\le p \}$.
However, this is not the case when $p\in \Ifree$, due to the fact that, in that case, we are only considering generators $x^p_w$ for $w\in E^0$ such that $[w]< p$.} 
 \end{remark}

\begin{proposition}
	\label{prop:MEC-theIsystem}
	Let $(E,C)$ be an adaptable separated graph, let $I=\Ifree \sqcup \Ireg$ be the canonical partition of $I=E^0/{\sim}$, and let $\mathcal J''$ be the $I$-system 
	of Definition \ref{def:Isystemofadpasepgraph}.
	Let $\mathbb P=\mathbb P_{{\rm free}} \sqcup \mathbb P_{{\rm reg}}$ 
	be the poset associated to $M(E,C)$, and $\mathcal J = (\mathbb P, \le , (G_p)_{p\in \mathbb P},\varphi_{p,q}\, (q<p))$ be the corresponding $\mathbb P$-system. 
	Then there is an isomorphism of systems $\mathcal J''\cong \mathcal J$. In particular 
	$$M(E,C) \cong M(\mathcal J) \cong M(\mathcal J'').$$
	\end{proposition}
	
	\begin{proof}
	Since $M(E,C)$ is a primely generated conical refinement monoid, there is a $\mathbb P$-system $\mathcal J$  such that $M(E,C)\cong M(\mathcal J)$. This system is described in detail in
	\cite[Section 2]{AP16}. We are going to follow that reference in order to identify the $\mathbb P$-system $\mathcal J$ with the $I$-system $\mathcal J''$.
	The first thing we do is to identify $\mathbb P$ with $I$.

	Let us define a relation $\vartriangleleft$ on $I$ as follows. For $p,q\in I$, set $p\vartriangleleft q$ if $p< q$ or $p=q\in \Ireg $. Observe that $\vartriangleleft$ is an antisymmetric and transitive relation on $I$. 
	Now define the monoid
	$M(I,\vartriangleleft)$ as the commutative monoid with family of generators $I$ and with relations $q=q+p$ if $p\vartriangleleft q$. The monoid $M(I, \vartriangleleft)$ is an antisymmetric finitely generated refinement monoid, and 
	its set of primes is precisely $I$. Moreover, the regular (resp. free) primes of $M(I,\vartriangleleft)$ are exactly the elements in $\Ireg $ (resp. $\Ifree$).  
	Now, it is straightforward, using the defining properties of an adaptable separated graph, to show that the antisymmetrization $\ol{M(E,C)}$ of $M(E,C)$
	is isomorphic to $M(I,\vartriangleleft )$, sending $\ol{a_v}\in \ol{M(E,C)}$ to $[v]\in M(I,\vartriangleleft )$.
	It follows from the description of the poset $\mathbb P$ associated to $M(E,C)$ given in \cite[p. 390, (1)]{AP16} and the above observations that $\mathbb P$, with its canonical partition 
	$\mathbb P = \mathbb P_{{\rm free}} \sqcup \mathbb P _{{\rm reg}}$, can be identified with $I$, and its partition $I=\Ifree \sqcup \Ireg $. Hence, the construction in  \cite[Section 2]{AP16}
	gives rise to an $I$-system $\mathcal J = (I,\le, \{G_p\}_{p\in I}, \varphi_{pq}\, (q<p) )$.  
	
	It remains to identify the groups $G_p$, for $p\in I$, and the maps $\varphi_{pq}\colon M_q\to G_p$ for $q<p$ (see \cite[Section 2]{AP16}).
	First, we observe that every hereditary subset of $E^0$ is $C$-saturated, because each $X\in C$ contains at least one loop. Therefore it follows from \cite[Corollary 6.10]{AG12} that the order-ideal 
	of $M(E,C)$ generated by a hereditary subset $H$ of $E^0$ is generated as a monoid by $\{ a_v :v\in H \}$.

	  Following \cite[Section 2]{AP16}, the group $G_p'$ is defined for each $p\in \Ifree $ to be the set
	$$\{ a_{v^p}+ \alpha : \alpha \in M(E,C) \quad \text{and}\quad a_{v^p}+\alpha \le a_{v^p} \}, $$
	endowed with the product $(a_{v^p}+\alpha)\circ (a_{v^p}+\beta ) = a_{v^p}+ ( \alpha + \beta )$. We want to show that $G_p''\cong G_p'$.
	To this end we define a map $\lambda_p \colon G_p'' \to G_p'$ by $\lambda_p (x^p_w) = a_{v^p}+ a_w$ for $w\in E^0$ with $[w]<p$. 
	Clearly, the defining relations of $G_p''$ are preserved by $\lambda_p$, so this assignment 
	defines a group homomorphism. Now if $\alpha \in M(E,C)$ and $a_{v^p}+\alpha \le a_{v^p}$, then $\alpha $ belongs to the order-ideal of $M(E,C)$ generated by $a_{v^p}$, and by the previous remark, it follows that $\alpha $
	must be a sum of elements of the form $a_w$ with $w \le v$. Now, it follows from the fact that $a_{v^p}$ is free that $\alpha $ is a sum of elements of the form $a_w$ with $w<v^p$. 
	This shows that $\lambda_p$ is surjective. In order to show that $\lambda _p$ is injective, let $\sum _{w\in A} n_w x^p_w - \sum _{w'\in B} m_{w'} x^p_{w'}$ be an element in the kernel of $\lambda_p$,
	where $A \cap B= \emptyset$, and $n_w, m_{w'} >0$. It then follows that $a_{v^p}+ \sum _{w\in A} n_w a_w = a_{v_p} + \sum _{w'\in B} m_{w'} a_{w'}$ in $M(E,C)$. Let $F$ be the free commutative monoid generated by $E^0$.
	It follows from Lemma \ref{lem:confluence} that there is $\gamma \in F$ such that $v^p+ \sum _{w\in A} n_w w \to \gamma $ and $v^p + \sum _{w'\in B} m_{w'} w' \to \gamma $  in $F$. 
	Note that $\gamma = v^p +\gamma '$, where $\gamma '= \sum _{w<v^p} l_ww$ for some $l_w\ge 0$. Now we transform $\sum _{w\in A} n_w x^p_w$ using corresponding steps to the ones used in the
	transformation  $v^p+ \sum _{w\in A} n_w w \to \gamma $, replacing each occurrence of
	a step $v^q\to_1 v^q + \sum _{j=1}^{g(q,i)} r(\beta (q,i,j)) $ for some $i=1,\dots , k(q)$ by the identity $0= \sum_{j=1}^{g(q,i)} x_{r(\beta (q,i,j))}^p $ in $G_p''$, for each $i=1,\dots , k(q)$, if $[w]= q\in \Ifree$ 
	and $q\le p$, and each occurrence of a step $w\to_1 \sum _{e\in s^{-1}(w)} r(e)$ by the identity  $x^p_w= \sum_{e\in s_E^{-1}(w)} x^p_{r(e)}$
	if $[w]\in \Ireg$ and $[w] < p$. By using this process, we arrive at the identity $\sum _{w\in A} n_w x^p_w = \sum _{w<v^p} l_w x^p_w$ in $G_p''$. With the same reasoning, we obtain
	$\sum _{w'\in B} m_{w'} x^p_{w'} = \sum _{w<v^p} l_w x^p_w$. So we get that $\sum _{w\in A} n_w x^p_w - \sum _{w'\in B} m_{w'} x^p_{w'} = 0$, as desired. 
	Finally the group $G_p$ is naturally isomorphic to $G_p'$ through the map $G_p'\to G_p$, $a_{v^p}+\alpha \mapsto (a_{v^p}+\alpha ) - a_{v^p} $ (\cite[Remark 2.5]{AP16}), so we get the isomorphism
	$G_p''\cong G_p$, which sends $x^p_w$ to  $(a_{v^p}+a_w ) - a_{v^p}$.

	If $p=[v]\in \Ireg$, then the archimedian component of $a_v$ in $M(E,C)$ is a group, and $G_p$ is defined to be this group, see \cite[Section 2]{AP16}.
	Let $e_p$ be the neutral element of $G_p$. Then one may check as before that the map $\lambda_p \colon G_p''\to G_p$ given by $x^p_w\mapsto e_p + a_w$ for $[w]\le p$, is a group isomorphism.  
	
	Finally it is straightforward to show that $\varphi_{p,q} \circ \widetilde{\lambda}_q = \lambda_p \circ \varphi''_{p,q}$ whenever $q<p$ in $I$, 
	where $\widetilde{\lambda}_q\colon M_q''\to M_q$ is the map induced by $\lambda_q$. Hence we get an isomorphism of $I$-systems $\mathcal J'' \cong \mathcal J$.
	Since $M(\mathcal J) \cong M(E,C)$ (\cite[Theorem 2.7]{AP16}), we get the last assertion in the statement.  
	\end{proof}

%%%%%%%%%%%%%%%%%%%%%%%%%%%%%%%%%%%%%%%5
%%%%%%%%%%%%%%%%%%%%%%%%%%%%%%%%%%%%%%%5

\subsection{Representing finitely generated refinement monoids}
In this subsection, given any finitely generated conical refinement monoid $M$, we build an adaptable separated graph $(E,C)$ such that its 
associated monoid is isomorphic to $M$. 

To this end recall from Section \ref{sec:Preliminaries} and \cite[Sections 1 and 2]{AP16} that, given any finitely generated conical refinement monoid $M$, one canonically associates to it 
an $I$-system $$\mathcal  J =  \left(I, \leq , (G_i)_{i\in I}, \varphi_{ji}\, (i<j)\right)$$ 
such that $M\cong M(\mathcal J)$ (\cite[Theorem 2.7]{AP16}). Moreover, the $I$-system $\mathcal J$ is finitely generated (meaning 
that $I$ is finite and all the abelian groups $G_i$ are finitely generated, \cite[Proposition 2.9]{AP16}). 

We now remind some terminology and facts concerning our monoids  before proving the main result of this section (see \cite{Dobb84, Brook, AP16,AP17} for background material).

%If $u$ is any element of an commutative monoid $M$, we define
%$$G'_u= \{ u+\alpha : \alpha \in M \text{ and } u+\alpha \le u \}.$$
%Then $G'_u$ is a group with respect to the operation $\circ $ given
%by:
%$$(u+\alpha) \circ (u+\beta ) = u+(\alpha +\beta )$$
%(see \cite[Definition 2.8 \& Lemma  2.9]{Brook}).

%It was shown in \cite[Lemma 2.4]{AP16} that if $p$ is a free prime in a primely generated refinement monoid, and $M_p$ denotes the archimedian component of $p$ in $M$, then there is a natural isomorphism $M_p\cong \N \times (G_p',\circ)$. To prove the main result of this section, we need a local version of this lemma.

%Assume that $M$ is a finitely generated conical refinement monoid, and let 
%$$\mathcal J = (I,\le , (G_p)_{p\in I}, \varphi_{p,q} \, (q<p))$$
%be its canonically associated $I$-system (\cite[Section 2]{AP16}), so that $M\cong M(\mathcal J)$. 
For $i\in I$, we define the {\it lower cover} $\rL (I,i)$ of $i$ in $I$ as 
 $$\rL (I,i):=\{j\in I\mid j<i \text{ and }[j,i]=\{j,i\}\}.$$
Let $p\in \Ifree $ and let  $\rL(I,p) = \{q_1, \dots , q_n \}$ be its lower cover. The archimedian component $M_p$ of $p$ has the form $M_{p}= \N\times G_p$ for the finitely generated abelian
group $G_p$. 

Using the notation established in \cite[Section 2]{AP17}, we denote by $J_p$ the lower subset of $I$ generated by $q_1,\dots ,q_n$, and let $M_{J_p}$ be the associated semigroup (cf. \cite[Corollary 2.4]{AP17}).
%By \cite[Lemma 2.10]{AP17}, the Grothendieck group of the order-ideal $M(\mathcal J_{J_p})$ associated to $J_p$ is precisely
%$\widetilde{G}_{J_p}= G(M_{J_p})$. 
Then, by \cite[Lemma 5.1]{AP17}, there is a surjective semigroup homomorphism
$$\varphi _p \colon M_{J_p} \to G_p$$
which is induced by the various maps $\varphi_{pq}$ for $q<p$. Consequently, we obtain a surjective group homomorphism
$G(\varphi _p) \colon G(M_{J_p}) \to G_p$. We say that an element $x$ in $G(M_{J_p})$ is {\it strictly positive} if it belongs to the
image of the canonical map $\iota _{J_p} \colon  M_{J_p} \to G(M_{J_p})$. We write $G(M_{J_p})^{++}= \iota_{J_p} (M_{J_p})$ for the set of strictly positive elements.

With the notation above, we provide the last proposition needed for Theorem \ref{thm:realM-by-sepgraph}.

\begin{proposition}
	\label{prop:generation-of-kernel}
	With the above notation and caveats, we have that the kernel of $G(\varphi_p)$ is generated by a finite number $x_1,\dots , x_k$ of strictly positive elements.
\end{proposition}

\begin{proof}
	Since $G(M_{J_p})$ is a finitely generated abelian group, we have that the kernel of $G(\varphi _p)$ is generated by a finite number of elements
	$y_1,\dots , y_m$.  So, it is enough to show that the subgroup generated by an element $y$ in the kernel of $G(\varphi _p)$ is contained in the subgroup generated by two
	strictly positive elements in the kernel of $G(\varphi_p)$.
	
	Recall that $\rL(I,p) = \{q_1, \dots , q_n \}$ is the lower cover of $p$. We assume that $q_1,\dots , q_r$ are free and that $q_{r+1},\dots , q_n$ are regular.  
	Now, let $y\in \ker (G(\varphi_p))$.  
	Using that the element $y$ can be expressed as a difference of two elements from $G(M_{J_p})^{++}$ and  \cite[Lemma 5.3]{AP17}, we see that there exist  
	positive integers $n_i, m_i$, $i=1,\dots ,r$, and elements $g_i\in G_{q_i}$, $i=1,\dots , n$, $h_j\in G_{q_j}$, $j=1,\dots , r$, such that
	$$ y=  \iota_{J_p} \Big( \sum_{i=1}^r \chi_{q_i} (n_i , g_i)  + \sum _{i=r+1}^n \chi_{q_i} (g_i)\Big) - \iota_{J_p} \Big( \sum_{j=1}^r \chi_{q_j} (m_j, h_j) \Big)$$
	Since $\varphi _p$ is surjective and $G(\varphi_p) (y) = 0$, there exists $z\in M_{J_p}$ such that
	$$ - \varphi_p \Big( \sum_{i=1}^r \chi_{q_i} (n_i , g_i)  + \sum _{i=r+1}^n \chi_{q_i} (g_i) \Big)  = - \varphi_p \Big(   \sum_{j=1}^r \chi_{q_j} (m_j, h_j) \Big) = \varphi_p (z).$$
	Therefore, if we define the elements $x_1= (\sum_{i=1}^r \chi_{q_i} (n_i , g_i)  + \sum _{i=r+1}^n \chi_{q_i} (g_i)) + z\in M_{J_p}$ and 
	$x_2 = (  \sum_{j=1}^r \chi_{q_j} (m_j, h_j) )  + z \in  M_{J_p}$, then we have $ \iota_{J_p}(x_1), \iota_{J_p} (x_2) \in \ker (\varphi_p)\cap G(M_{J_p})^{++}$, and $y = \iota_{J_p}(x_1)-\iota_{J_p}(x_2)$.
	This shows the result.
\end{proof}

\begin{theorem}
	\label{thm:realM-by-sepgraph}
	Let $M$ be a finitely generated refinement monoid, and let $\mathcal J$ be the associated $I$-system, so that $M\cong M(\mathcal J)$. Then there is an adaptable separated graph $(E,C)$  such that
	$$M(E,C) \cong M(\mathcal J) \cong M.$$
\end{theorem}

\begin{proof}
	The proof follows the lines of the proof of \cite[Proposition 5.13]{AP17}. This result says that, if the natural map
	$G(\varphi_p ) \colon G(M_{J_p})  \to G_p$ is an almost isomorphism for every free prime $p$, then there is a
	row-finite directed graph $E$ such that $M\cong M(E)$. (In particular, this holds if every prime in $M$ is regular).
	We will only outline the point in which the proof has to be adapted, recalling some of the relevant notation.

	The proof works by induction. Assume that $J$ is a lower subset of $I$ and that an adaptable separated graph $(E_J, C^J)$ of the desired form has been constructed so that there is
	a monoid isomorphism $$\gamma _J \colon M(J) \to M(E_J, C^J),$$
	where $M(J)$ is the order-ideal of $M$ generated by $J$, sending the canonical semigroup generators to the corresponding sets of vertices, as specified in \cite[p. 113]{AP17}.
	In case $J\ne I$, let $p$ be a minimal element of $I\setminus J$, and write $J' = J\cup \{ p \}$. If $p$ is a regular prime or $p$ is minimal, proceed as in
	the proof of \cite[Proposition 5.13]{AP17}.
	
	Assume that $p$ is a non-minimal free prime. By Proposition \ref{prop:generation-of-kernel}, there are a finite number of strictly positive elements
	$x_1,\dots , x_k$ which generate the kernel of the map $G(\varphi_p)$. Now, using the same arguments as in the proof of \cite[Proposition 5.13]{AP17}, we may find
	elements $\widehat{x}_i\in M (E_J, C^J) $, $i=1, \dots , k$, which are non-negative integer combinations of the vertices of $E_J$ such that $\gamma _J (x_i)= \widehat{x}_i$
	for $i=1,\dots , k$. Observe that $\widehat{x}_i\in H_{\gamma _J(J_p)}$, so that we may consider its class (denoted in the same way) in $M_{\gamma_J (J_p)}$.
	Now, we introduce the adaptable separated graph $(E_{J'},C^{J'})$. We define $E_{J'}^0 = E_J^0 \sqcup \{v^p \}$, and $C^{J'} \setminus C^{J'}_{v^p} = C^J $, that is, the structure of $(E_{J'},C^{J'})$ is the same
	as the structure of $(E_J, C^J)$ when restricted to the vertices of $E_J$. For the new vertex $v^p$ we define $C^{J'}_{v^p} = \{ X_1^{(p)},\dots , X_k^{(p)}  \}$, where each $X_i^{(p)}$ has the form 
	described in Definition \ref{def:adaptable-sepgraphs}(3), and the edges $\alpha (p,i), \beta (p,i,t)$, $t=1,\dots , g(p,i)$ are chosen in such a way that the relations
	\begin{equation}
	\label{eq:arrows-for-free}
	v^p = v^p+ \widehat{x}_i
	\end{equation}
	are satisfied in the graph monoid $M(E_{J'}, C^{J'})$, for $i=1,\dots ,k$.
	(Here we set $k(p)= k$).

	By Proposition \ref{prop:MECisprimelygenerated}, $M(E_{J'},C^{J'})$ is a primely generated conical refinement monoid. Its corresponding system has been determined 
	in Proposition \ref{prop:MEC-theIsystem}.
	In particular, we know that the set of primes of
	$M(E_{J'},C^{J'})$ is $\mathbb P (M(E_{J},C^{J})) \cup \{ v^p \}$ and that $v^p$ is a free prime in $M(E_{J'},C^ {J'})$. 
	Consequently, we have that the archimedian component $M(E_{J'},C^{J'})[v^p]$ of $M(E_{J'},C^{J'})$ at $v^p$ satisfies
	$$M(E_{J'},C^{J'})[v^p] = \N \times G'_{v^p}$$
	for some abelian group $G'_{v^p}$, and that the map
	$\phi_p : M(E_{J'},C^{J'})_{\gamma _J(J_p)} \to G'_{v^p}$ induced by the various
	semigroup homomorphisms
	$$
	\begin{array}{crcl}
	\phi_q^p\colon & M(E_{J'},C^{J'})_{\gamma _J (q)}  & \rightarrow   & G'_{v^p}  \\
	& y & \mapsto  &   (v^p + y ) - v^p
	\end{array}
	$$
	for $q<p$ is surjective. So, we obtain a surjective group homomorphism
	$$G(\phi _p ) \colon  G(M(E_{J'},C^{J'})_{\gamma_J(J_p)}) \to G'_{v^p}.$$
	In order to simplify the notation, we will write $M(E_{J'},C^{J'})_{J_p}$ instead of $M(E_{J'},C^{J'})_{\gamma _J(J_p)}$.
	
	It is readily seen that the natural map $M(E_J,C^J) \to M(E_{J'},C^{J'})$ defines a monoid isomorphism from $M(E_J,C^J)$ onto
	an order-ideal of $M(E_{J'},C^{J'})$; hence, we will identify $M(E_J,C^J)$ with its image without further comment.
	Moreover, the component  $M(E_{J'},C^{J'})_{J_p}$ clearly coincides with the component
	$M(E_J, C^J)_{J_p}$.
	
	Now, the monoid isomorphism $\gamma _J \colon M(J) \to M(E_J,C^J)$ restricts to a semigroup isomorphism
	$M_{J_p} \to M(E_J,C^J)_{J_p}$, which induces a group isomorphism
	$$\widetilde{\gamma}_{J_p} \colon G(M_{J_p}) \to G(M(E_{J},C^J)_{J_p})$$
	of the Grothendieck groups.
	Set $K:= \ker (G(\phi_p ))$, and notice that the relation (\ref{eq:arrows-for-free})
	implies that $\widetilde{\gamma}_{J_p}  (x_i)=\widehat{x}_i \in K $ for $i=1,\dots, k$.
	
	Hence, there is a commutative diagram with exact rows
	\begin{equation}\label{diagram:cap-avall}
	\begin{CD}
	0 @>>> \langle x_1,\dots ,x_k \rangle  @>>> G(M_{J_p}) @>{G(\varphi_p )}>> G_p @>>> 0 \\
	& &    @VVV  @VV{\widetilde{\gamma}_{J_p}}V  @VV{ \gamma_ p}V \\
	0 @>>> K @>>> G(M(E_J,C^J)_{J_p}) @>{G(\phi_p )}>> G'_{v^p} @>>> 0 \, \, ,
	\end{CD}
	\end{equation}
	where $\gamma _p \colon G_p \to G'_{v_p}$ is the map induced from the cokernel of the inclusion $\langle x_1,\dots ,x_k \rangle \hookrightarrow G(M_J)$ to the cokernel of the inclusion
	$K \hookrightarrow G(M(E_J,C^J)_{J_p})$. Notice that $\gamma_p$ is an onto map.
	
	We now define the map
	$$ \gamma _{J'} \colon M(J') \to M(E_{J'},C^{J'})$$
	extending the monoid isomorphism $\gamma _J: M(J)\rightarrow M(E_J,C^J)$, and defining $\gamma_{J'}$ on the component $M_p\cong \N\times G_p$ of $M(J')$ by the formula
	$$\gamma_{J'}( mp + g) = m v^p + \gamma_p (g) $$
	for $m\in \N$ and $g\in G_p$.
	By \cite[Corollary 1.8]{AP16}, to show that $\gamma_{J'}$ is a well-defined monoid homomorphism, it suffices  to show that
	if $q< p$ and $y\in \AC{M}{q}=M_q$ then $\gamma _{J'}(y) + \gamma _{J'}(p) = \gamma_{J'}(\varphi_{p,q}(y) + p)$, that is,
	$\gamma _{J}(y) + v^p = \gamma_{p}(\varphi_{p,q}(y))+ v^p$. For $x\in M_{J_p}$, we may define a map
	$$\tau_q \colon M_q \to G(M_{J_p})$$
	by $\tau _q(y) = (x+y) -x \in G(M_{J_p})$. The map $\tau_q$ is a semigroup homomorphism and does not depend on the particular choice of $x\in M_{J_p}$.
	Moreover, we have 
	$\varphi_{p,q} = G(\varphi _p) \circ \tau _q$. Analogously, we have a map
	$$\tau_{\gamma_J(q)} \colon M(E_{J'},C^{J'})_{\gamma_J(q)} \to  G(M(E_{J'},C^{J'})_{J_p}) = G(M(E_{J},C^J)_{J_p})$$
	such that  $\phi^{v^p}_{\gamma_J(q)} = G(\phi_p ) \circ \tau _{\gamma_J(q)}$, and clearly
	$\widetilde{\gamma} _{J_p} \circ \tau_q = \tau_{\gamma _J (q)} \circ \gamma_J|_{M_q} $.
	
	Using this fact, and the commutativity of (\ref{diagram:cap-avall}), we have that
	\begin{align*}
		\gamma _p (\varphi_{p,q} (y)) + v^p & = \gamma _p (G(\varphi _p) (\tau_q  (y))) + v^p \\
		&  =  G(\phi_p) ( \widetilde{\gamma} _{J_p} (\tau _q (y))) + v^p \\
		& =  G(\phi_p) ( \tau_{\gamma_J (q)} (\gamma_J(y))) + v^p \\
		& = \phi^{v^p}_{ \gamma_J(q)} (\gamma _J(y)) + v^p  \\
		& = ((v^p + \gamma_J (y) )-v^p) + v^p  \\
		& = v^p + \gamma _J (y) \, ,
	\end{align*}
	as desired.
	
	This shows that there is a well-defined monoid homomorphism 
	$$\gamma_{J'} \colon M(J') \to M(E_{J'},C^{J'})$$ 
	sending the canonical semigroup generators of $M(J')$ to the corresponding canonical
	sets of vertices
	seen in $M(E_{J'}, C^{J'})$. In particular, $\gamma_{J'} $ is an onto map.
	
	In order to prove the injectivity of $\gamma_{J'} $, we can build an inverse map $\delta_{J'}:  M(E_{J'},C^{J'}) \to M(J')$,
	as follows: on $M(E_{J},C^J)$ we define $\delta_{J'}$ to be $\gamma_J^{-1}$, while $\delta_{J'}(v^p):=p$. Notice that the only
	relations on $M(E_{J'},C^{J'})$ not occurring already in $M(E_J, C^J)$ are $v^p=v^p+\widehat{x}_i$, $i=1,\dots ,k$, where $\gamma_J(x_i)=\widehat{x}_i$. Thus, $\delta_{J'}(\widehat{x}_i)=x_i$.
	But $x_1,\dots , x_k$ generate the kernel of the map
	$$G(\varphi_p): G(M_{J_p})\rightarrow G_p\hookrightarrow \widehat{G}_p=\Z\times G_p,$$
	so that $(p+x_i)-p$ equals $0$ in $\widehat{G}_p$. Hence, the relations $p=p+x_i$ hold in $M(J')$, for $i=1,\dots ,k$. Thus, $\delta_{J'}$ is a well-defined monoid
	homomorphism, and it is the inverse of $\gamma_{J'}$. This completes the proof of the inductive step.
\end{proof}
\section*{Acknowledgments}

This research project was initiated when the authors were at the Centre de Recerca Matem\`atica as part of the Intensive Research Program \emph{Operator algebras: dynamics and interactions} in 2017, 
and the work was significantly supported by the research environment and facilities provided there. The authors thank the Centre de Recerca Matem\`atica for its support.

\end{document}